\documentclass[10pt, reqno]{amsart}
\usepackage{graphicx, amssymb, amsmath, amsthm}
\numberwithin{equation}{section}

\usepackage{color}
\usepackage{epsfig}
\usepackage{subfigure}
\usepackage{tikz}
\let\Re=\undefined\DeclareMathOperator*{\Re}{Re}

\newcommand{\R}{\mathbb{R}}
\newcommand{\C}{\mathbb{C}}

\newcommand{\Z}{\mathbb{Z}}

\newtheorem{theorem}{Theorem}[section]

\newtheorem{lemma}[theorem]{Lemma}

\newtheorem{proposition}[theorem]{Proposition}

\theoremstyle{definition}

\newtheorem{remark}[theorem]{Remark}

\makeatletter
\newcommand{\Extend}[5]{\ext@arrow0099{\arrowfill@#1#2#3}{#4}{#5}}
\makeatother

\begin{document}
\title[dispersive estimates]{Decay estimates for one Aharonov-Bohm solenoid in a uniform magnetic field I: Schr\"odinger equation}

\author{Haoran Wang}
\address{Department of Mathematics, Beijing Institute of Technology, Beijing 100081;}
\email{wanghaoran@bit.edu.cn}

\author{Fang Zhang}
\address{Department of Mathematics, Beijing Institute of Technology, Beijing 100081;}
\email{zhangfang@bit.edu.cn;}

\author{Junyong Zhang}
\address{Department of Mathematics, Beijing Institute of Technology, Beijing 100081;}
\email{zhang\_junyong@bit.edu.cn; }

\begin{abstract}
This is the first of a series of papers in which we investigate the decay estimates for dispersive equations with Aharonov-Bohm solenoids in a uniform magnetic field.  In this first starting paper, we prove the local-in-time dispersive estimates and Strichartz estimates for Schr\"odinger equation with one Aharonov-Bohm solenoid  in a uniform magnetic field. The key ingredient is the construction of Schr\"odinger propagator, we provide two methods to construct the propagator. The first one is  combined the strategies of \cite{FFFP1} and \cite{GYZZ22, FZZ22}, and the second one is based on the Schulman-Sunada formula in sprit of \cite{stov, stov1} in which the heat kernel has been studied. In future papers, we will continue investigating this quantum model for wave with one or  multiple Aharonov-Bohm solenoids in a uniform magnetic field.

\end{abstract}

%\bigskip\bigskip
\maketitle

\section{Introduction}
Let us consider the electromagnetic Hamiltonian
\begin{equation*}
H_{A, V}=-(\nabla+iA(x))^2+V(x),
\end{equation*}
where the electric scalar potential $V: \R^n\to \R$ and the magnetic vector potential
\begin{equation*}
A(x)=(A^1(x),\ldots, A^n(x)): \, \R^n\to \R^n
\end{equation*}satisfies the Coulomb gauge condition
\begin{equation}\label{div0}
\mathrm{div}\, A=0.
\end{equation}
In three dimensions, the magnetic vector potential $A$ produces a magnetic field $B$, which is given by
\begin{equation*}
B=\mathrm{curl} (A)=\nabla\times A.
\end{equation*}
In general dimensions $n\geq2$, $B$ should be viewed as the matrix-valued field $B:\R^n\to \mathcal{M}_
{n\times n}(\R)$ given by
\begin{equation}\label{B-n}
B:=DA-DA^t,\quad B_{ij}=\frac{\partial A^i}{\partial x_j}-\frac{\partial A^j}{\partial x_i}.
\end{equation}
The Schr\"odinger operators with electromagnetic potentials have been extensively studied from the aspects of
spectral and scattering theory, we refer to Avron-Herbst-Simon \cite{AHS1,AHS2,AHS3} and Reed-Simon \cite{RS}, in which many important physical potentials (e.g. the constant magnetic field and the Coulomb electric potential) are discussed. The purpose of our program here is to study how the electric or magnetic potentials affect the short-time or long-time behavior of the solutions for
dispersive equations (e.g. the Schr\"odinger, wave and Klein-Gordon).\vspace{0.2cm}

Indeed, the program on the study of the decay estimates and Strichartz estimates of the dispersive equations has a long history due to their significance in analysis and PDEs fields. We refer to   \cite{BPSS, BPST, CS,DF, DFVV, EGS1, EGS2, S} and the references therein for the the classical and important Schr\"odinger and wave equations with electromagnetic potentials in mathematical and physical fields. However, since the different potentials has different effects, it is hard to provide a uniform argument to treat all of potentials, and the picture of the program is too far to completed, especially for some critical physic potentials. For example, the dispersive equations with the Aharonov-Bohm potential, as a diffraction physic model and scaling critical purely magnetic potential, has attracted more and more people to study from mathematic viewpoint.
In \cite{FFFP, FFFP1}, Fanelli, Felli, Fontelos, and Primo studied the validity of the time-decay estimate for the Schr\"odinger equation associated with the operator included the Aharonov-Bohm potential.
However, the argument in \cite{FFFP, FFFP1} breaks down for wave equation due to the lack of  pseudoconformal invariance (which was used for Schr\"odinger equation). Very recently, Fanelli, Zheng and the last author \cite{FZZ22}
established the Strichartz estimate for wave equation by constructing the odd sin propagator. To solve open problems, raised  in the survey \cite{Fanelli}, about the the dispersive estimate for other equations (e.g. Klein-Gordon), Gao, Yin, Zheng and the last author \cite{GYZZ22} constructed the spectral measure and further proved the time decay and Strichartz estimates of Klein-Gordon equation.  The potential models in \cite{FFFP, FFFP1,FZZ22,GYZZ22}  are scaling invariant and with no unbounded (at infinity) perturbations, which is a special cases of \eqref{H-A} below with $B_0\equiv 0$.\vspace{0.2cm}

In this paper, as a sequence of recent paper \cite{GYZZ22, FZZ22}, we study the decay and Strichartz estimates for the Schr\"odinger equation in a mixed magnetic fields on the plane pierced by one infinitesimally thin Aharonov-Bohm solenoid and subjected to a perpendicular uniform magnetic field of constant magnitude $B_0$.
More precisely, we focus on the typical 2D purely magnetic model Hamiltonians
\begin{equation}\label{H-A}
H_{\alpha,B_0}=-(\nabla+i(A_B(x)+A_{\mathrm{hmf}}(x)))^2,
\end{equation}
where $A_B(x)$ is the Aharonov-Bohm potential in \cite{AB59}
\begin{equation}\label{AB-potential}
A_B(x)=\alpha\Big(-\frac{x_2}{|x|^2},\frac{x_1}{|x|^2}\Big),\quad x=(x_1,x_2)\in\mathbb{R}^2\setminus\{0\},
\end{equation}
with $\alpha\in\mathbb{R}$ representing the circulation of $A_B$ around the solenoid
and $A_{\mathrm{hmf}}(x)$ is the potential
\begin{equation}\label{A-hmf}
A_{\mathrm{hmf}}(x)=\frac{B_0}{2}(-x_2,x_1),\quad B_0>0.
\end{equation}
From \eqref{B-n}, the total magnetic filed is given by $B_0+\alpha \delta(x)$ where $\delta$ is the Dirac delta function.
From mathematics viewpoint, it is worth emphasizing three key features that about the perturbations of the potentials here. The first one is that the Aharonov-Bohm potential producing the singular magnetic field has the same scaling of $\nabla$ which is homogenous degree of $-1$, so that the perturbation of Aharonov-Bohm potential \eqref{AB-potential} is non-trivial; the second one is that the
magnetic poetential \eqref{A-hmf} is homogenous degree of $1$ so that the Schr\"odinger operator \eqref{H-A} is not scaling invariant so is the Schr\"odinger equation.
The potential $A_{\mathrm{hmf}}(x)$ is unbounded and the magnetic filed $B(x)=B_0$ from \eqref{B-n} is a uniform magnetic field caused a trapped well. On one hand, due to the perturbations of
magnetic fields \eqref{A-hmf} are involved, the operator $H_{\alpha,B_0}$ has pure point spectrum, the dispersive behavior of Schr\"odinger equation associated with $H_{\alpha,B_0}$  becomes quite different from the model in \cite{FZZ22, GYZZ22}.  On the other hand, because of the Aharonov-Bohm effect,  a feature of the Mehler kernel which is related to the Schr\"odinger kernel associated with pure uniform magnetic field breaks down. Indeed, if $\alpha=0$,
the Schr\"odinger kernel can be written as
\begin{equation*}
e^{itH_{0,B_0}}(x,y)=\frac{B_0}{4\pi\sin(B_0t)}\exp\Big\{\frac{B_0}{4i}\big(\cot(B_0t)|x-y|^2-2x\wedge y\big)\Big\},
\end{equation*}
furthermore, one can write
\begin{equation*}
\begin{split}
e^{itH_{0,B_0}}(x,y)=\frac{B_0}{4\pi\sin(B_0t)}&\exp\Big\{\frac{B_0}{4i}\cot(B_0t)\big(|x|^2+|y|^2\big)\Big\}\\
&\times\exp\Big\{i\frac{B_0y\cdot R(B_0t)x}{2\sin(B_0t)}\Big\},
\end{split}
\end{equation*}
where $R(\theta)$ is the usual $2\times 2$ rotation matrix given by
\begin{equation*}
R(\theta)=\left(\begin{array}{cc}\cos\theta &-\sin\theta\\ \sin\theta & \cos\theta \end{array}\right).
\end{equation*}
Therefore, see \cite[Theorem 2]{KL}, one can prove the Strichartz estimates
\begin{equation*}
\|e^{itH_{0,B_0}}f\|_{L^q_t((0,\frac{\pi}{B_0});L^r_x(\R^2)}=(4\pi)^{1-\frac4q}\|e^{it\Delta}f\|_{L^q_t(\R;L^r_x(\R^2)}\leq C\|f\|_{L^2(\R^2)},
\end{equation*}
where $(q,r)\in\Lambda_S$ defined in \eqref{adm:S} below. \vspace{0.2cm}

More precisely,  we aim to study the dispersive behavior of the magnetic Schr\"odinger equation
\begin{equation}\label{equ:S}
\begin{cases}
i\partial_t u(t,x)- H_{\alpha,B_0} u(t,x)=0,\\
u(0,x)=u_0(x),
\end{cases}
\end{equation}
where $H_{\alpha,B_0}$ is given in \eqref{H-A} with the perturbation of potentials $A_B(x)$ and $A_{\mathrm{hmf}}(x)$.\vspace{0.2cm}

 %%%%%%%%%%%%%%%%%%%%
%%%%%MAIN THEOREMS%%%%%%
%%%%%%%%%%%%%%%%%%%%

Now we state our main results.
\begin{theorem}\label{thm:S} Let  $H_{\alpha,B_0}$ be  in \eqref{H-A} with the potentials being \eqref{AB-potential}-\eqref{A-hmf}
and let $u(t,x)$ be the solution of \eqref{equ:S}.
Then there exists a constant $C>0$ such that the dispersive estimate
\begin{equation}\label{dis-S}
\|u(t,x)\|_{L^\infty(\R^2)}\leq C |\sin(tB_0)|^{-1}\|u_0\|_{L^1(\R^2)},\quad t\neq \frac{k\pi}{B_0},\, k\in\Z,
\end{equation}
and the Strichartz estimate holds for
\begin{equation}\label{str-S}
\|u(t,x)\|_{L^q([0,T];L^p(\R^2))}\leq C \|u_0\|_{L^2(\R^2)},
\end{equation}
 where $T\in (0, \frac{\pi}{2B_0})$ and $(q,p)\in \Lambda_S$ with
\begin{equation}\label{adm:S}
\Lambda_S:=\Bigg\{(q,p)\in[2,+\infty]\times [2,+\infty):\frac2q=2\Big(\frac12-\frac1p\Big)\Bigg\}.
\end{equation}
\end{theorem}
\begin{remark} The decay estimate \eqref{dis-S} is the periodic with period $\pi/B_0$.
 The endpoint of the time interval $T$ in the Strichartz estimates \eqref{str-S} depends on the coefficient $B_0$ of unbounded potential. Actually,
 if the unbounded potentials disappear, i.e.  $B_0=0$, one can take $T$ to be $+\infty$ safely, that is, the global-in-time Strichartz estimates, which is corresponding to the Laplacian with the Aharonov-Bohm potential (see Theorem 1.3 of \cite{FFFP1}). However, in the present paper, as mentioned above, we only obtain local-in-time Strichartz estimates associated with the operator \eqref{H-A} due to the unbounded potentials caused trapped well.
\end{remark}

\begin{remark} Let
\begin{equation*}
A(x)=A_B(x)+A_{\mathrm{hmf}}(x),
\end{equation*}
one can verify that $\mathrm{div}\, A=0$ satisfies \eqref{div0}. Then, one observe that
\begin{equation*}
\begin{split}
H_{\alpha,B_0}&=-(\nabla+i(A_B(x)+A_{\mathrm{hmf}}(x)))^2
\\&=-\Delta+\big(\frac{B_0}2\big)^2 |x|^2+\frac{\alpha^2}{|x|^2} +i B_0(-x_2,x_1)\cdot \nabla +i \frac{2\alpha}{|x|^2}(-x_2,x_1)\cdot \nabla+\alpha B_0.
\end{split}
\end{equation*}
One will see that the operator is perturbed by the inverse-square potential and harmonic oscillator. This phenomenon is natural because the unbounded potential cause a trapped well, the energy cannot be dispersive for long time.  This is closely relate to the models with pure harmonic oscillators, i.e. $H_{0,V}=-\Delta+|x|^2$, in which Koch and Tataru \cite{KT05} proved the decay is the periodic with period $\pi$
\begin{equation*}
\|e^{it H_{0,V}}\|_{L^1(\R^n)\to L^\infty(\R^n)}\leq C|\sin t|^{-\frac n2}.
\end{equation*}
\end{remark}

 \vspace{0.2cm}

The paper is organized as follows. In Section \ref{sec:pre}, in a preliminary step, we recall the the self-adjoint extension of the operator $H_{\alpha,B_0}$, and study the spectrum of $H^{F}_{\alpha, B_0}$ (which is the Friedrichs of $H_{\alpha,B_0}$). In Section \ref{sec:const1}, we construct the Schr\"odinger propagator by combing the strategies of \cite{FFFP1} and \cite{GYZZ22,FZZ22}. In Section \ref{sec:const2}, we construct the Schr\"odinger propagator by using another method based on the Schulman-Sunada formula.  Finally, in Section \ref{sec:proof}, we prove the Theorem \ref{thm:S} by using the representation of the Schr\"odinger propagator constructed in previous sections.
\vspace{0.2cm}

{\bf Acknowledgments:}\quad  The authors thank L. Fanelli and P. \v{S}t'ov\'{\i}\v{c}ek for helpful discussions. This work is supported by National Natural Science Foundation of China (12171031, 11901041, 11831004).
\vspace{0.2cm}

\section{preliminaries} \label{sec:pre}

In this section, we first recall the self-adjoint extension of the operator $H_{\alpha,B_0}$, the Friedrichs extension, and then we study the spectrum of $H^{F}_{\alpha,B_0}$ (which is the Friedrichs of $H_{\alpha,B_0}$).

\subsection{Quadratic form and the self-adjoint extension}

Define the space $\mathcal{H}_{\alpha,B_0}^1(\R^2)$ as the completion of $\mathcal{C}_c^\infty(\R^2\setminus\{0\};\C)$
with respect to the norm
\begin{equation*}
\|f\|_{\mathcal{H}_{\alpha,B_0}^1(\R^2)}=\Big(\int_{\R^2}|\nabla_{\alpha,B_0} f(x)|^2 dx\Big)^{\frac12}
\end{equation*}
where
\begin{equation*}
\nabla_{\alpha,B_0} f(x)=\nabla f+i(A_B+A_{\mathrm{hmf}})f.
\end{equation*}
The quadratic form $Q_{\alpha,B_0}$ associated with $H_{\alpha,B_0}$ is defined by
\begin{equation*}
\begin{split}
Q_{\alpha,B_0}: & \quad \quad \mathcal{H}_{\alpha,B_0}^1\to \R\\
Q_{\alpha,B_0}(f)&=\int_{\R^2}|\nabla_{\alpha,B_0} f(x)|^2dx.
\end{split}
\end{equation*}
Then the quadratic form $Q_{\alpha,B_0}$ is positive definite, thus it implies that the operator
$H_{\alpha,B_0}$ is symmetric semi bounded from below which admits a self-adjoint extension (Friedrichs extension)
$H^{F}_{\alpha,B_0}$
with the natural form domain
\begin{equation*}
\mathcal{D}=\Big\{f\in \mathcal{H}_{\alpha,B_0}^1(\R^2):  H^{F}_{\alpha,B_0}f\in L^{2}(\R^2)\Big\}
\end{equation*}
Even though the operator $H_{\alpha,B_0}$ has many other self-adjoint extensions (see \cite{ESV}) by the von Neumann extension theory,  in this whole paper, we use the simplest Friedrichs extension and briefly write $H_{\alpha,B_0}$ as its Friedrichs extension $H^{F}_{A}$.

\subsection{The spectrum of the operator $H_{\alpha,B_0}$}

In this subsection, we modify the argument of \cite{FFFP1} to obtain the eigenvalue and eigenfunction of the
Schr\"odinger operator
\begin{equation*}
H_{\alpha,B_0}=-(\nabla+i(A_B(x)+A_{\mathrm{hmf}}(x)))^2,
\end{equation*}
where the magnetic vector potentials are in \eqref{AB-potential} and
\eqref{A-hmf}.
More precisely, we will prove that
\begin{proposition}[The spectrum for $H_{\alpha,B_0}$]\label{prop:spect}
Let  $H_{\alpha,B_0}$ be the self-adjoint Schr\"odinger operator in \eqref{H-A}.
Then the eigenvalues of $H_{\alpha,B_0}$ are discrete and are given by
\begin{equation}\label{eigen-v}
\lambda_{k,m}=(2m+1+|k+\alpha|)B_0+(k+\alpha)B_0,\quad m,\,k\in\mathbb{Z},\, m\geq0,
\end{equation}
and the (finite) multiplicity of $\lambda_{k,m}$ is
\begin{equation*}
\#\Bigg\{j\in\mathbb{Z}:\frac{\lambda_{k,m}-(j+\alpha)B_0}{2B_0}-\frac{|j+\alpha|+1}{2}\in\mathbb{N}\Bigg\}.
\end{equation*}
Furthermore, let $\theta=\frac{x}{|x|}$, the corresponding eigenfunction is given by
\begin{equation}\label{eigen-f}
V_{k,m}(x)=|x|^{|k+\alpha|}e^{-\frac{B_0 |x|^2}{4}}\, P_{k,m}\Bigg(\frac{B_0|x|^2}{2}\Bigg)e^{ik\theta}
\end{equation}
where  $P_{k,m}$ is the polynomial of degree $m$ given by
\begin{equation*}
P_{k,m}(r)=\sum_{n=0}^m\frac{(-m)_n}{(1+|k+\alpha|)_n}\frac{r^n}{n!}.
\end{equation*}
with $(a)_n$ ($a\in\R$) denoting the Pochhammer's symbol
\begin{align*}
(a)_n=
\begin{cases}
1,&n=0;\\
a(a+1)\cdots(a+n-1),&n=1,2,\cdots
\end{cases}
\end{align*}\end{proposition}

\begin{remark} One can verify that the orthogonality holds
$$\int_{\R^2}V_{k_1,m_1}(x) V_{k_2,m_2}(x) \,dx=0,\quad \text{if}\quad (k_1, m_1)\neq (k_2, m_2).$$
\end{remark}

\begin{remark} Let $L^\alpha_m(t)$ be the generalized Laguerre polynomials
\begin{equation*}
L^\alpha_m(t)=\sum_{n=0}^m (-1)^n \Bigg(
  \begin{array}{c}
    m+\alpha \\
    m-n \\
  \end{array}
\Bigg)\frac{t^n}{n!},
\end{equation*}
and the well known orthogonality relation
\begin{equation*}
\int_0^\infty x^{\alpha} e^{-x}L^\alpha_m(x) L^\alpha_n (x)\, dx=\frac{\Gamma(n+\alpha+1)}{n!} \delta_{n,m},\end{equation*}
where $\delta_{n,m}$ is the Kronecker delta. Let $\tilde{r}=\frac{B_0|x|^2}{2}$ and $\alpha_k=|k+\alpha|$, then
\begin{equation}\label{P-L}
P_{k,m}(\tilde{r})=\sum_{n=0}^m\frac{(-1)^n m(m-1)\cdots(m-(n-1))}{(\alpha_k +1)(\alpha_k+2)\cdots(\alpha_k+n)}\frac{\tilde{r}^n}{n!}=\Bigg(
  \begin{array}{c}
    m+\alpha_k \\
    m \\
  \end{array}
\Bigg)^{-1}L^{\alpha_k}_m(\tilde{r}).
\end{equation}
Therefore,
\begin{equation}\label{V-km-2}
\|V_{k,m}(x)\|^2_{L^2(\R^2)} =\pi \Big(\frac{2}{B_0}\Big)^{\alpha_k+1}\Gamma(1+\alpha_k) \Bigg(
  \begin{array}{c}
    m+\alpha_k \\
    m \\
  \end{array}
\Bigg)^{-1}. \end{equation}
\end{remark}

\begin{remark}
Recall the Poisson kernel formula for Laguerre polynomials \cite[(6.2.25)]{AAR01}: for $a, b, c, \alpha>0$
\begin{equation}\label{La-po}
\begin{split}
&\sum_{m=0}^\infty e^{-cm}\frac{m !}{\Gamma(m+\alpha+1)} L_m^{\alpha} (a) L_m^{\alpha} (b)\\
&=\frac{e^{\frac{\alpha c}2}}{(ab)^{\frac{\alpha}2}(1-e^{-c})} \exp\left(-\frac{(a+b)e^{-c}}{1-e^{-c}}\right) I_{\alpha}\left(\frac{2\sqrt{ab}e^{-\frac{c}2}}{1-e^{-c}}\right)
\end{split}
\end{equation}
then this together with \eqref{P-L} gives
\begin{equation}\label{Po-L}
\begin{split}
&\sum_{m=0}^\infty e^{-cm}\frac{m !}{\Gamma(m+\alpha_k+1)}
\Bigg(\begin{array}{c}
    m+\alpha_k \\
    m \\
  \end{array}
\Bigg)^2 P_{k,m} (a) P_{k,m} (b)\\
&=\frac{e^{\frac{\alpha_k c}2}}{(ab)^{\frac{\alpha_k}2}(1-e^{-c})} \exp\left(-\frac{(a+b)e^{-c}}{1-e^{-c}}\right) I_{\alpha_k}\left(\frac{2\sqrt{ab}e^{-\frac{c}2}}{1-e^{-c}}\right).
\end{split}
\end{equation}
\end{remark}

\begin{proof}
Notice that the operator \eqref{H-A}, in the polar coordinates $(r,\theta)$,  has a nice representation
\begin{equation*}
H_{\alpha, B_0}=-\partial_r^2-\frac{1}{r}\partial_r+\frac{1}{r^2}\Big(-i\partial_\theta+\alpha+\frac{B_0r^2}{2}\Big)^2.
\end{equation*}
We want to solve the eigenfunction equation
\begin{equation}\label{eigen-p}
H_{\alpha, B_0} g(x)=\lambda g(x)
\end{equation}
in the domain \footnote{Here we use the Friedrichs self-adjoint extension. } of $H_{A, V}$.
Define the projectors $P_k$ onto the eigenspaces of the angular momentum as
\begin{equation*}
P_kf(r,\theta)=\frac{1}{2\pi}\int_0^{2\pi}e^{ik(\theta-\theta')}f(r,\theta')d\theta':=f_k(r) e^{ik\theta},\quad k\in\Z
\end{equation*}
then it clear that the operator $H_{\alpha, B_0}$ commutes with the projectors $P_k$.
In the polar coordinates $(r,\theta)$, \eqref{eigen-p} implies that for every $k\in\Z$,
\begin{equation}\label{eq:gk}
g''_k(r)+\frac{1}{r}g'_k(r)-\frac{1}{r^2}\Big(k+\alpha+\frac{B_0 r^2}{2}\Big)^2g_k(r)=-\lambda g_k(r),
\end{equation}
where
\begin{equation*}
g_k(r)= \frac{1}{2\pi}\int_0^{2\pi}e^{-ik\theta}g(r,\theta)d\theta.
\end{equation*}
Let
\begin{equation*}
\phi_k(s)=\Big(\frac{2s}{B_0}\Big)^{-\frac{|k+\alpha|}{2}}e^{\frac{s}{2}}g_k\Big(\sqrt{\frac{2s}{B_0}}\Big),
\end{equation*}
 then $\phi_k$ satisfies
\begin{equation}\label{eq:phik}
s\phi''_k(s)+(1+|k+\alpha|-s)\phi'_k(s)-\frac{1}{2}\Big(1+|k+\alpha|+k+\alpha-\frac{\lambda}{B_0}\Big)\phi_k(s)=0.
\end{equation}

\begin{lemma}\label{lem:KCH}
The Kummer Confluent Hypergeometric equation
\begin{equation*}
s\phi''(s)+(b-s)\phi'(s)-a\phi(s)=0,\quad s>0,
\end{equation*}
has two linearly independent solutions given by the Kummer function (also called confluent hypergeometric function)
\begin{equation*}
M(a,b,s)=\sum_{n=0}^\infty\frac{(a)_n}{(b)_n}\frac{s^n}{n!},\quad b\neq0,-1,-2,\cdots
\end{equation*}
where the Pochhammer's symbol
\begin{align*}
(a)_n=
\begin{cases}
1,&n=0;\\
a(a+1)\cdots(a+n-1),&n=1,2,\cdots
\end{cases}
\end{align*}
and the Tricomi function (also called confluent hypergeometric function of the second kind)
\begin{equation*}
U(a,b,s)=\frac{\Gamma(1-b)}{\Gamma(a-b+1)}M(a,b,s)+\frac{\Gamma(b-1)}{\Gamma(a)}s^{1-b}M(a-b+1,2-b,s).
\end{equation*}
\end{lemma}

By using Lemma \ref{lem:KCH}, we have two linearly independent solutions of \eqref{eq:phik}, hence two linearly independent solutions of \eqref{eq:gk} are given by
\begin{align*}
g_k^1(\lambda;r)&=r^{|\alpha+k|}M\Big(\beta(k,\lambda),\gamma(k),\frac{B_0r^2}{2}\Big)e^{-\frac{B_0r^2}{4}}\\
g_k^2(\lambda;r)&=r^{|\alpha+k|}U\Big(\beta(k,\lambda),\gamma(k),\frac{B_0r^2}{2}\Big)e^{-\frac{B_0r^2}{4}}
\end{align*}
with
\begin{align*}
\beta(k,\lambda)&=\frac{1}{2}\Big(1+k+\alpha+|k+\alpha|-\frac{\lambda}{B_0}\Big),\\
\gamma(k)&=1+|k+\alpha|.
\end{align*}
Therefore, the general solution of \eqref{eq:phik} is given by
\begin{align}\label{sol:gk}
g_k(r)=&A_kg_k^1(\lambda;r)+B_kg_k^2(\lambda;r)\\
&=r^{|\alpha+k|}e^{-\frac{B_0r^2}{4}}
\Bigg(A_kM\Big(\beta(k,\lambda),\gamma(k),\frac{B_0r^2}{2}\Big)+B_kU\Big(\beta(k,\lambda),\gamma(k),\frac{B_0r^2}{2}\Big)\Bigg),\nonumber
\end{align}
where $A_k,B_k$ are two constants which are dependent on $k$.

\begin{lemma}[\cite{AS65},Chap.13] \label{lem:asy}
The following properties hold:
\begin{itemize}
\item The two functions $M(a,b,z)$ and $U(a,b,z)$ are linearly dependent if and only if $a\in-\mathbb{Z}_+$.

\item $M(a,b,z)$ is an entire function of $z$ and it is regular at $z=0$. However, $U(a,b,z)$ is singular at the origin provided that $b>1$ and $a\notin-\mathbb{Z}_+$ and it holds true that
    \begin{equation}\label{U-0}
    \lim_{z\rightarrow0^+}z^{b-1}U(a,b,z)=\frac{\Gamma(b-1)}{\Gamma(a)}.
    \end{equation}
    If $b\in(1,2)$, the asymptotic behavior of $U(a,b,z)$ as $z\rightarrow0^+$ is
    \begin{equation*}
    U(a,b,z)=\frac{\Gamma(1-b)}{\Gamma(a-b+1)}+\frac{\Gamma(b-1)}{\Gamma(a)}z^{1-b}+O(z^{2-b}).
    \end{equation*}
\item If $-a\notin\mathbb{N}$, the asymptotic behavior as $z\rightarrow+\infty$ holds true:
    \begin{equation}\label{M-inf}
    M(a,b,z)=\frac{\Gamma(b)}{\Gamma(b-a)}(-z)^{-a}(1+O(z^{-1}))+\frac{\Gamma(b)}{\Gamma(a)}e^z z^{a-b}(1+O(z^{-1})),
    \end{equation}
    and
    \begin{equation*}
    U(a,b,z)=z^{-a}(1+O(z^{-1})).
    \end{equation*}
\end{itemize}
\end{lemma}

Now we use this asymptotic lemma to conclude more detail information about the eigenvalues and eigenfunctions.
Let
$$m:=-\beta(k,\lambda)=-\frac{1}{2}\Big(1+k+\alpha+|k+\alpha|-\frac{\lambda}{B_0}\Big).$$
On one hand, from \eqref{M-inf}, if $m\notin \mathbb{N}$, the function
\begin{equation*}
M\big(-m,\gamma(k),\tilde{r}\big)\sim e^{ \tilde{r}} \tilde{r}^{-m-\gamma(k)},\quad \tilde{r}\rightarrow+\infty,
\end{equation*}
 is singular at $+\infty$; while if $m\in\mathbb{N}=\{0,1,2,\cdots\}$, then $M(m,\gamma(k),\tilde{r})$ is in fact a polynomial of degree $m$ in $\tilde{r}$, which we shall denote as $P_{k,m}$, i.e.,
\begin{equation*}
P_{k,m}(\tilde{r})=M(-m,1+|k+\alpha|,\tilde{r})=\sum_{n=0}^m\frac{(-m)_n}{(1+|k+\alpha|)_n}\frac{\tilde{r}^n}{n!}.
\end{equation*}
On the other hand, note that $\gamma(k)=1+|k+\alpha|\geq1$, from \eqref{U-0}, it follows that
\begin{equation*}
U\big(\beta(k,\lambda),\gamma(k),\tilde{r}\big)\sim \tilde{r}^{-|k+\alpha|},\quad \tilde{r}\rightarrow0+,
\end{equation*}
with the implicit constant depending only on $\lambda$ and $k$.
Hence, by letting $\tilde{r}=\frac{B_0r^2}{2}$ and using \eqref{sol:gk} and Lemma \ref{lem:asy}, we have for fixed $k\in\mathbb{Z}$ that
\begin{equation*}
g_k(r)\sim B_k r^{|\alpha+k|}e^{-\frac{B_0r^2}{4}}
U\Big(\beta(k,\lambda),\gamma(k),\frac{B_0r^2}{2}\Big)\sim B_k r^{-|\alpha+k|},\quad\text{as}\quad r\rightarrow0+.
\end{equation*}
and
\begin{equation*}
g_k(r)\sim A_k r^{|\alpha+k|}e^{-\frac{B_0 r^2}{4}}
M\Big(\beta(k,\lambda),\gamma(k),\frac{B_0 r^2}{2}\Big)\sim A_k e^{\frac{B_0r^2}{4}}r^{-1+k+\alpha-\frac{\lambda}{B_0}},\quad\text{as}\quad r\rightarrow \infty.
\end{equation*}

We now conclude that one must have $B_k\equiv0$. Indeed, otherwise,
from the fact that the eigenfunction $g\in D(H_{\alpha,B_0})$, thus we have that
\begin{equation*}
\int_0^\infty g_k^2(r)\frac{dr}{r}\leq\int_{\mathbb{R}^2}\frac{g^2(x)}{|x|^2}dx<\infty,
\end{equation*}
which is obviously contradict with the fact that the integral $\int_0^\infty r^{-2|k+\alpha|-1} dr$ is divergent for all $k\in\Z$ at $0$.\vspace{0.2cm}

We next conclude that one must have $A_k\equiv0$ if $m\notin\mathbb{N}$. Actually,
from the fact that the eigenfunction $g\in D(H_{\alpha,B_0})$ again, thus we have that
\begin{equation*}
\int_0^\infty g_k^2(r)\, rdr\leq\int_{\mathbb{R}^2} g^2(x)dx<\infty.
\end{equation*}
However, we note that $B_0>0$ and
\begin{equation*}
\int_0^\infty g_k^2(r)\, rdr\geq \int_1^\infty e^{\frac{B_0r^2}{4}}\, rdr
\end{equation*}
which is divergent. This is a contradiction. \vspace{0.2cm}

Therefore, we must have
$$\mathbb{N} \ni m=-\frac{1}{2}\Big(1+k+\alpha+|k+\alpha|-\frac{\lambda}{B_0}\Big),$$
and
$$g_k(r)=r^{|k+\alpha|}e^{-\frac{B_0r^2}{4}}\, P_{k,m}\Big(\frac{B_0r^2}{2}\Big).$$
Therefore, we prove the function
\begin{equation*}
V_{k,m}(x)=|x|^{|k+\alpha|}e^{-\frac{B_0|x|^2}{4}}\, P_{k,m}\Big(\frac{B_0|x|^2}{2}\Big)e^{ik\theta},
\end{equation*}
belongs to $D(H_{\alpha,B_0})$, thus providing an eigenfunction of the operator $H_{\alpha,B_0}$.
Thus from $-m=\frac{1}{2}\Big(1+k+\alpha+|k+\alpha|-\frac{\lambda}{B_0}\Big)$, we solve \eqref{eigen-p} to obtain the eigenvalues $\lambda$ of $H_{\alpha,B_0}$
\begin{equation*}
\lambda_{k,m}=(2m+1+|k+\alpha|)B_0+(k+\alpha)B_0,\quad k\in\mathbb{Z},\, m\in\mathbb{N}.
\end{equation*}
\end{proof}

\section{The Schr\"odinger propagator}\label{sec:const1}

In this section, we construct the Schr\"odinger propagator by using the spectrum property of Proposition \ref{prop:spect}. The combined strategies are  from \cite{FFFP1} and \cite{GYZZ22,FZZ22}.\vspace{0.2cm}

More precisely, we will prove the following result.
\begin{proposition}\label{prop:S}
Let $H_{\alpha,B_0}$ be the operator in \eqref{H-A} and suppose $x=r_1(\cos\theta_1,\sin\theta_1)$
and $y=r_2(\cos\theta_2,\sin\theta_2)$. Let $u(t,x)$ be the solution of the Schr\"odinger equation
\begin{equation*}
\begin{cases}
\big(i\partial_t-H_{\alpha,B_0}\big)u(t,x)=0,\\
u(0,x)=f(x).
\end{cases}
\end{equation*}
Then
\begin{equation*}
u(t,x)=e^{-itH_{\alpha,B_0}}f=\int_{\R^2} K_S(x,y) f(y)\, dy,
\end{equation*}
where $t\neq \frac{k\pi}{B_0},\, k\in\Z$.  Let $\theta=\theta_1-\theta_2-tB_0\in\mathbb{R},$  there exists an integer $j_0$ satisfying $\theta+2j_0\pi\in [-\pi,\pi]$.
Define
\begin{equation*}
\chi(\theta, j_0)=
\begin{cases}
\begin{array}{ll}
                       1, & \hbox{if\, $|\theta+2j_0\pi|<\pi$ ;} \\
                       e^{-i2\pi \alpha}+1 , & \hbox{if\, $\theta+2j_0\pi=-\pi;$}\\
                       e^{i2\pi\alpha}+1, & \hbox{if \,$\theta+2j_0\pi=\pi$.}
                     \end{array}
\end{cases}
\end{equation*}
Then the kernel of Schr\"odinger propagator $e^{-itH_{\alpha,B_0}}$ has the representation
\begin{equation}\label{S:express}
\begin{split}
K_S(x,y)&=\frac{B_0e^{-itB\alpha}}{8\pi^2 i\sin (tB_0)}e^{\frac{iB_0(r_1^2+r_2^2)}{4\tan (tB_0)}}\\
\times&\Big[e^{\frac{B_0r_1r_2}{2i\sin (tB_0)}\cos(\theta_1-\theta_2-tB_0)}
e^{-i\alpha(\theta_1-\theta_2-tB_0+2j_0\pi)}\chi(\theta, j_0)\\
&-\frac{\sin(\pi\alpha)}{\pi}\int_{\R} e^{-\frac{B_0r_1r_2}{2i\sin (tB_0)}\cosh s}\frac{e^{-\alpha s}}{1+e^{-s+i(\theta_1-\theta_2-tB_0)}}\,ds\Big].
\end{split}
\end{equation}

\end{proposition}

\begin{proof}
We construct the representation formula for the kernel of the Schr\"odinger flow $e^{-itH_{\alpha,B_0}}$ by combining the argument of \cite{FFFP1} and \cite{FZZ22,GYZZ22}.
\vspace{0.2cm}

Our starting point is the Proposition \ref{prop:spect}. Let $\tilde{V}_{k,m}$ be the $L^2$-normalization of $V_{k,m}$ in \eqref{eigen-f}, then the eigenfunctions
$\Big\{\tilde{V}_{k,m}\Big\}_{k\in\Z, m\in\mathbb{N}} $
form an orthonormal basis of $L^2(\mathbb{R}^2)$ corresponding to the eigenfunctions of $H_{\alpha,B_0}$. \vspace{0.2cm}

We expand the initial data $f(x)\in L^2$ as
\begin{equation*}
f(x)=\sum_{k\in\Z, \atop m\in\mathbb{N}} c_{k,m}\tilde{V}_{k,m}(x)
\end{equation*}
where
\begin{equation}\label{cmk1}
c_{k,m}=\int_{\mathbb{R}^2}f(x)\overline{\tilde{V}_{k,m}(x)}\, dx.
\end{equation}
The solution $u(t,x)$ of \eqref{equ:S} can be written as
\begin{equation}\label{sol:expan}
u(t,x)=\sum_{k\in\Z, \atop m\in\mathbb{N}} u_{k,m}(t)\tilde{V}_{k,m}(x),
\end{equation}
where $u_{k,m}(t)$ satisfies the ODE
\begin{equation*}
\begin{cases}
&iu'_{k,m}(t)=\lambda_{k,m}u_{k,m}(t),\\
& u_{k,m}(0)=c_{k,m},\quad k\in\Z, \, m\in\mathbb{N}.
\end{cases}
\end{equation*}
Thus we obtain $u_{k,m}(t)=c_{k,m}e^{-it\lambda_{k,m}}$.
Therefore the solution \eqref{sol:expan} becomes
\begin{equation*}
u(t,x)=\sum_{k\in\Z, \atop m\in\mathbb{N}} c_{k,m}e^{-it\lambda_{k,m}}\tilde{V}_{k,m}(x).
\end{equation*}
Plugging \eqref{cmk1} into the above expression yields
\begin{equation*}
u(t,x)=\sum_{k\in\Z, \atop m\in\mathbb{N}} e^{-it\lambda_{k,m}}\left(\int_{\mathbb{R}^2}f(y)\overline{\tilde{V}_{k,m}(y)} dy\right)\tilde{V}_{k,m}(x).
\end{equation*}
We write $f$ in a harmonic spherical expansion
\begin{equation*}
f(y)=\sum_{k\in\Z} f_k(r_2) e^{ik\theta_2},
\end{equation*}
where
\begin{equation}\label{f-k}
f_k(r_2)=\frac{1}{2\pi}\int_0^{2\pi} f(r_2, \theta_2) e^{-ik\theta_2}\, d\theta_2,\quad r_2=|y|,
\end{equation}
we thus have
\begin{align*}
u(t,x)&=\sum_{k\in\Z, \atop m\in\mathbb{N}} e^{-it\lambda_{k,m}}\frac{V_{k,m}(x)}{\|V_{k,m}\|^2_{L^2}}\Bigg(\int_0^\infty f_k(r_2)e^{-\frac{B_0r_2^2}{4}}\, P_{k,m}\Big(\frac{B_0r_2^2}{2}\Big) r_2^{1+\alpha_k}\mathrm{d}r_2\Bigg)\\
&=\Big(\frac{B_0}{2\pi}\Big)\sum_{k\in\Z}e^{ik\theta_1}\frac{B_0^{\alpha_k}e^{-it\beta_k}}{2^{\alpha_k}\Gamma(1+\alpha_k)}\Bigg[\sum_{m=0}^\infty
\Bigg(
  \begin{array}{c}
    m+\alpha_k \\
    m \\
  \end{array}
\Bigg)e^{-2itmB_0}\\
&\times\Bigg(\int_0^\infty f_k(r_2)(r_1r_2)^{\alpha_k}e^{-\frac{B_0(r_1^2+r_2^2)}{4}}P_{k,m}\left(\frac{B_0r_2^2}{2}\right)P_{k,m}\left(\frac{B_0r_1^2}{2}\right)r_2 \mathrm{d}r_2\Bigg)\Bigg],
\end{align*}
where $\alpha_k=|k+\alpha|$ and we use \eqref{V-km-2}, \eqref{eigen-v}, \eqref{eigen-f} and
\begin{equation*}
\begin{split}
\lambda_{k,m}&=(2m+1+|k+\alpha|)B_0+(k+\alpha)B_0\\
&:=2mB_0+\beta_k
\end{split}
\end{equation*}
with $\beta_k=(1+|k+\alpha|)B_0+(k+\alpha)B_0>0$.

Notice that $P_{k,m}$ can be expressed as (see e.g. \cite[(6.2.15)]{AAR01})
\begin{equation*}
P_{k,m}\left(\frac{r^2}{2}\right)=\frac{\Gamma(1+\alpha_k)}{\Gamma(1+\alpha_k+m)}e^{\frac{r^2}{2}}r^{-\alpha_k}2^{\frac{\alpha_k}{2}}\int_0^\infty
e^{-s}s^{m+\frac{\alpha_k}{2}}J_{\alpha_k}(\sqrt{2s}r)ds,
\end{equation*}
in terms of Bessel functions $J_{\alpha_k}$ of order $\alpha_k$.

Hence, we have
\begin{align*}
u(t,x)=&\frac{B_0}{2\pi}\sum_{k\in\mathbb{Z}}e^{ik\theta_1}\frac{B_0^{\alpha_k}e^{-it\beta_k}}{2^{\alpha_k}\Gamma(1+\alpha_k)}\Bigg[\sum_{m\in\mathbb{N}}
\left(
  \begin{array}{c}
    m+\alpha_k \\
    m \\
  \end{array}
\right)e^{-2itmB_0}\\
&\times\Big(\frac{\Gamma(1+\alpha_k)}{\Gamma(1+\alpha_k+m)}\Big)^2\Bigg(\int_0^\infty f_k(r_2)(r_1r_2)^{\alpha_k} e^{-\frac{B_0(r_1^2+r_2^2)}{4}}e^{\frac{B_0(r_1^2+r_2^2)}{2}}\Big(\frac{2}{B_0r_2^2}\Big)^{\frac{\alpha_k}{2}}\\
&\times\Big(\frac{2}{B_0r_1^2}\Big)^{\frac{\alpha_k}{2}}\Big(\int_0^\infty\int_0^\infty e^{-s_1-s_2}(s_1s_2)^{m+\frac{\alpha_k}{2}}
J_{\alpha_k}(\sqrt{2B_0s_1}r_1)J_{\alpha_k}(\sqrt{2B_0s_2}r_2)ds_1ds_2\Big)r_2dr_2\Bigg)\Bigg]\\
=&\frac{B_0}{2\pi}\sum_{k\in\mathbb{Z}}e^{ik\theta_1}e^{-it\beta_k}\Gamma(1+\alpha_k)\Bigg[\sum_{m\in\mathbb{N}}
\left(
  \begin{array}{c}
    m+\alpha_k \\
    m \\
  \end{array}
\right)\frac{e^{-2itmB_0}}{(\Gamma(1+\alpha_k+m))^2}\\
&\times\Bigg(\int_0^\infty f_k(r_2)e^{\frac{B_0(r_1^2+r_2^2)}{4}}\Big(\int_0^\infty\int_0^\infty e^{-s_1-s_2}(s_1s_2)^{m+\frac{\alpha_k}{2}}\\
&\times J_{\alpha_k}(\sqrt{2B_0s_1}r_1)J_{\alpha_k}(\sqrt{2B_0s_2}r_2)ds_1ds_2\Big)r_2dr_2\Bigg)\Bigg],\quad \big(\text{variable changes:}\, s_1\to s_1^2,\, s_2\to s_2^2 \big)\\
=&\frac{2B_0}{\pi}\sum_{k\in\mathbb{Z}}e^{ik\theta_1}e^{-it\beta_k}\Bigg[\int_0^\infty f_k(r_2)e^{\frac{B_0(r_1^2+r_2^2)}{4}}e^{i\alpha_k(tB_0+\frac{\pi}{2})}\\
&\times\Bigg(\int_0^\infty\int_0^\infty\frac{s_1s_2}{e^{s_1^2+s_2^2}}\Bigg(\sum_{m\in\mathbb{N}}
\frac{(-1)^me^{-i(tB_0+\frac{\pi}{2})(2m+\alpha_k)}}{\Gamma(1+m)\Gamma(1+\alpha_k+m)}(s_1s_2)^{2m+\alpha_k}\Bigg)\\
&\times J_{\alpha_k}\big(\sqrt{2B_0}s_1r_1\big)J_{\alpha_k}\big(\sqrt{2B_0}s_2 r_2\big)ds_1ds_2\Bigg)r_2dr_2\Bigg].
\end{align*}
Since (see e.g. \cite[(4.5.2)]{AAR01})
\begin{equation*}
\sum_{m=0}^\infty\frac{(-1)^me^{-i(tB_0+\frac{\pi}{2})(2m+\alpha_k)}}{\Gamma(1+\alpha_k+m)\Gamma(1+m)}(s_1s_2)^{2m+\alpha_k}
=J_{\alpha_k}(2s_1s_2e^{-i(tB_0+\frac{\pi}{2})}),
\end{equation*}
then we have
\begin{align*}
u(t,x)=\frac{2B_0}{\pi}\sum_{k\in\mathbb{Z}}e^{ik\theta_1}e^{it(\alpha_kB_0-\beta_k)+i\alpha_k\frac{\pi}{2}}
\int_0^\infty e^{\frac{B_0(r_1^2+r_2^2)}{4}}f_k(r_2)G_{k,t}(r_1,r_2)r_2dr_2
\end{align*}
with
\begin{equation*}
\begin{split}
G_{k,t}(r_1,r_2)=\int_0^\infty\int_0^\infty&\frac{s_1s_2}{e^{s_1^2+s_2^2}}J_{\alpha_k}(2s_1s_2e^{-i(tB_0+\frac{\pi}{2})})
\\&\times J_{\alpha_k}(\sqrt{2B_0}r_1s_1)J_{\alpha_k}(\sqrt{2B_0}r_2s_2)ds_1ds_2.
\end{split}
\end{equation*}
Using formula (see \cite[formula (1), P.395]{Wat44})
\begin{align}\label{formula}
\int_0^\infty&e^{-p^2t^2}J_\nu(at)J_\nu(bt)tdt=\frac{1}{2p^2}e^{-\frac{a^2+b^2}{4p^2}}I_\nu\Big(\frac{ab}{2p^2}\Big),\nonumber\\
&\qquad \Re\nu>-1,\quad |\arg p|<\frac{\pi}{4},\quad I_\nu(r)=e^{-\frac{1}{2}\nu\pi i}J_\nu(re^{i\frac{\pi}{2}}),
\end{align}
with $t=s_2,p=1,a=\sqrt{2B_0}r_2,b=2s_1e^{-i(tB_0+\frac{\pi}{2})},\nu=\alpha_k$, we get
\begin{align*}
\int_0^\infty e^{-s_2^2}J_{\alpha_k}\big(\sqrt{2B_0}r_2s_2\big)&J_{\alpha_k}\Big(2s_1s_2e^{-i(tB_0+\frac{\pi}{2})}\Big)s_2ds_2\\
&=\frac{1}{2}e^{-\frac{B_0r_2^2+2s_1^2e^{-i(2tB_0+\pi)}}{2}}I_{\alpha_k}\big(\sqrt{2B_0}r_2s_1e^{-i(tB_0+\frac{\pi}{2})}\big),
\end{align*}
where $I_{\alpha_k}$ denotes the modified Bessel function of order $\alpha_k$.
Hence
\begin{align*}
G_{k,t}(r_1,r_2)&=\frac{1}{2}\int_0^\infty e^{-s_1^2}J_{\alpha_k}(\sqrt{2B_0}r_1s_1)e^{-\frac{B_0r_2^2+2s_1^2e^{-i(2tB_0+\pi)}}{2}}
I_{\alpha_k}\big(\sqrt{2B_0}r_2s_1e^{-i(tB_0+\frac{\pi}{2})}\big)s_1ds_1\\
&=\frac{1}{4B_0}\int_0^\infty e^{-\frac{s_1^2}{2B_0}}J_{\alpha_k}(r_1s_1)e^{-\frac{B_0r_2^2+\frac{s_1^2}{B_0}e^{-i(2tB_0+\pi)}}{2}}
I_{\alpha_k}(r_2s_1e^{-i(tB_0+\frac{\pi}{2})})s_1ds_1\\
&=\frac{1}{4B_0}e^{-i\alpha_k\frac{\pi}{2}}e^{-\frac{B_0r_2^2}{2}}\int_0^\infty e^{-\frac{s_1^2}{2B_0}(1+e^{-i(2tB_0+\pi)})}J_{\alpha_k}(r_1s_1)J_{\alpha_k}(r_2s_1e^{-itB_0})s_1ds_1\\
&=\frac{1}{4B_0}e^{-i\alpha_k\frac{\pi}{2}}e^{-\frac{B_0r_2^2}{2}}\frac{B_0}{1+e^{-i(2tB_0+\pi)}} e^{-\frac{B_0r_1^2+B_0r_2^2e^{-2itB_0}}{2(1+e^{-i(2tB_0+\pi)})}}I_{\alpha_k}(\frac{B_0r_1 r_2e^{-itB_0}}{1+e^{-i(2tB_0+\pi)}})\\
&=\frac{1}{4(1+e^{-i(2tB_0+\pi)})}e^{-i\alpha_k\frac{\pi}{2}} e^{-\frac{B_0(r_1^2+r_2^2)}{2(1+e^{-i(2tB_0+\pi)})}}I_{\alpha_k}\Big(\frac{B_0r_1 r_2e^{-itB_0}}{1+e^{-i(2tB_0+\pi)}}\Big)
\end{align*}
where we have used the fact $I_\nu(r)=e^{-i\frac{\pi}{2}\nu}J_\nu(re^{i\frac{\pi}{2}})$ in the third equality and \eqref{formula} with $t=s_1,p^2=\frac{1+e^{-i(2tB_0+\pi)}}{2B_0},a=r_1,b=r_2e^{-itB_0},\nu=\alpha_k$ in the last equality.
Finally, we use the simple fact (for $tB_0\neq k\pi, k\in\mathbb{Z}$)
\begin{equation*}
\frac{1}{1+e^{-i(2tB_0+\pi)}}=\frac{e^{itB_0}}{2i\sin tB_0}
\end{equation*}
to obtain
\begin{align*}
u(t,x)=&\frac{2B_0}{\pi}\frac{e^{itB_0}}{8i\sin tB_0}\sum_{k\in\mathbb{Z}}e^{ik\theta_1}e^{it(\alpha_kB_0-\beta_k)}\\
&\times\Bigg[\int_0^\infty e^{\frac{B_0(r_1^2+r_2^2)}{4}}f_k(r_2)e^{-\frac{B_0(r_2^2+r_1^2)}{4}\cdot\frac{e^{itB_0}}{i\sin tB_0}}I_{\alpha_k}\Bigg(\frac{B_0r_1r_2}{2i\sin (tB_0)}\Bigg)r_2dr_2\Bigg]\\
=&\frac{B_0e^{-itB_0\alpha}}{4\pi i\sin (tB_0)}\sum_{k\in\mathbb{Z}}e^{ik\theta_1}e^{-itB_0k}\Bigg(\int_0^\infty
f_k(r_2)e^{-\frac{B_0(r_1^2+r_2^2)\cos (tB_0)}{4i\sin (tB_0)}}I_{\alpha_k}\Bigg(\frac{B_0r_1r_2}{2i\sin (tB_0)}\Bigg)r_2dr_2\Bigg),
\end{align*}
where we use $\beta_k=(k+\alpha)B_0+(1+\alpha_k)B_0$ in the last line.
Recalling $f_k$ in \eqref{f-k}, we finally have
\begin{equation}\label{S:sum}
\begin{split}
u(t,x)
=&\frac{B_0e^{-itB_0\alpha}}{8\pi^2 i\sin (tB_0)}\int_0^\infty \int_0^{2\pi}e^{-\frac{B_0(r_1^2+r_2^2)}{4i\tan (tB_0)}} \\&\times\sum_{k\in\mathbb{Z}}
\Bigg( e^{ik(\theta_1-\theta_2-tB_0)}
I_{\alpha_k}\bigg(\frac{B_0r_1r_2}{2i\sin (tB_0)}\bigg)\Bigg) f(r_2,\theta_2)r_2dr_2 d\theta_2.
\end{split}
\end{equation}
Now we consider the summation in $k$.
Let $z=\frac{B_0r_1r_2}{2i\sin (tB_0)}$ and $\theta=\theta_1-\theta_2-tB_0$ and recall the following integral representation \cite{Wat44} of the modified Bessel function $I_\nu$
\begin{equation}\label{mBessel:integral}
I_\nu(z)=\frac{1}{\pi}\int_0^\pi e^{z\cos s}\cos(\nu s)d s-\frac{\sin(\pi\nu)}{\pi}\int_0^\infty e^{-z\cosh s}e^{-s\nu}ds.
\end{equation}
Recall $\alpha_k=|\alpha+k|$ and $\alpha\in(0,1)$, we need to consider
\begin{equation}\label{equ:term1}
 \frac1\pi \sum_{k\in\Z} e^{ik\theta} \int_0^\pi e^{z\cos s} \cos(\alpha_k s) ds,
\end{equation}
  and
    \begin{equation}\label{equ:term2}
 \frac1\pi \sum_{k\in\Z} e^{ik\theta } \sin(\pi\alpha_k)\int_0^\infty e^{-z\cosh s} e^{-s\alpha_k} ds.
  \end{equation}
Similarly as in \cite[Proposition 3.1]{FZZ22}, we use the Poisson summation formula
\begin{equation*}
\sum_{j\in\Z} \delta(x-2\pi j)=\sum_{k\in\Z} \frac1{2\pi} e^{i kx},
\end{equation*}
to  obtain
 \begin{equation*}
 \begin{split}
 \frac1{\pi} \sum_{k\in\Z} e^{ik\theta} \cos(\alpha_k s) =
 \sum_{j\in\Z}\big[e^{is\alpha}\delta(\theta+s+2j\pi)
 +e^{-is\alpha}\delta(\theta-s +2j\pi)\big].
 \end{split}
 \end{equation*}
Therefore we have
\begin{align}
&  \eqref{equ:term1}\nonumber\\
&=\sum_{j\in\mathbb{Z}}\int_0^\pi e^{z\cos s}\Bigg(e^{is\alpha}\delta(\theta+s+2\pi j)+e^{-is\alpha}\delta(\theta-s+2\pi j)\Bigg)d s\nonumber\\
&=e^{z\cos(\theta)}e^{-i\alpha\theta}\sum_{j\in\mathbb{Z}} \chi_{[-\pi,\pi]}(\theta+2\pi j)e^{-i2\pi j\alpha} \label{term1+2}.
\end{align}
Recall $\theta=\theta_1-\theta_2-tB_0\in \mathbb{R}$ and observe that
$$\theta+2\pi j_0\in (-\pi, \pi)\implies \theta+2\pi (j_0\pm m) \notin (-\pi, \pi), |m|\ge 1, $$
the summation in $j$ is only one term (except for $\theta+2j_0\pi=\pm\pi$). It's easily to verify that the identity \eqref{term1+2} are the same at $\theta+2j_0\pi=\pm\pi.$ For all $\theta\in\mathbb{R},$ there exists $j_0\in\mathbb{Z}$ such that $\theta+2j_0\pi\in[-\pi,\pi].$ Therefore, we get
\begin{align*}
 \eqref{equ:term1}=e^{z\cos(\theta)}e^{-i\alpha\theta}e^{-i2\pi j_0\alpha}\left\{
                     \begin{array}{ll}
                       1, & \hbox{$|\theta+2j_0\pi|<\pi$ ;} \\
                       e^{-i2\pi \alpha}+1 , & \hbox{$\theta+2j_0\pi=-\pi;$}\\
                       e^{i2\pi\alpha}+1, & \hbox{$\theta+2j_0\pi=\pi$.}
                     \end{array}
                   \right.
\end{align*}
Recall again that $\alpha_k=|\alpha+k|,k\in\mathbb{Z}$ and $\alpha\in(0,1)$, we have
\begin{align*}
\alpha_k=|\alpha+k|=
\begin{cases}
k+\alpha,&k\geq0;\\
-\alpha-k,&k\leq-1,
\end{cases}
\end{align*}
and hence
\begin{align*}
\sin(\pi\alpha_k)=\sin(\pi|\alpha+k|)=
\begin{cases}
\cos k\pi\sin(\pi\alpha)=e^{ik\pi}\sin(\pi\alpha),&k\geq0;\\
-\cos k\pi\sin(\pi\alpha)=-e^{ik\pi}\sin(\pi\alpha),&k\leq-1.
\end{cases}
\end{align*}
We can compute the summation as follows

\begin{align*}
&\sum_{k\in\mathbb{Z}}\sin(\pi|\alpha+k|)e^{-s|\alpha+k|}e^{ik\theta}\\
&\quad=\sin(\pi\alpha)\sum_{k\geq0}e^{ik\pi}e^{-s(k+\alpha)}e^{ik\theta}-\sin(\pi\alpha)\sum_{k\leq-1}e^{ik\pi}e^{s(k+\alpha)}e^{ik\theta}\\
&\quad=\sin(\pi\alpha)\Big(e^{-\alpha s}\sum_{k\geq0}e^{ik(is+\pi+\theta)}-e^{\alpha s}\sum_{k\geq 1} e^{ik(is-\pi-\theta)}\Big)\\
&\quad=\sin(\pi\alpha)\Big(\frac{e^{-\alpha s}}{1-e^{i(is+\pi+\theta)}}-\frac{e^{\alpha s}e^{i(is-\pi-\theta)}}{1-e^{i(is-\pi-\theta)}}\Big)\\
&\quad=\sin(\pi\alpha)\Big(\frac{e^{-\alpha s}}{1+e^{-s+i\theta}}+\frac{e^{\alpha s}}{1+e^{s+i\theta)}}\Big).
\end{align*}
Therefore, we see
\begin{align*}
 & \eqref{equ:term2}\\
 = &\frac{\sin(\alpha\pi)}{\pi}\int_0^\infty e^{-z\cosh s}
\Big(\frac{e^{-\alpha s}}{1+e^{-s+i\theta}}+\frac{e^{\alpha s}}{1+e^{s+i\theta)}}\Big)\,ds\\
 = &\frac{\sin(\alpha\pi)}{\pi}\int_{-\infty}^\infty e^{-z\cosh s}
\frac{e^{-\alpha s}}{1+e^{-s+i\theta}}\,ds.
\end{align*}

So we put the two terms together and recall $\theta=\theta_1-\theta_2-tB_0$
to obtain
\begin{equation*}
\begin{split}
K_S(x,y)&=\frac{B_0e^{-itB\alpha}}{8\pi^2 i\sin (tB_0)}e^{\frac{iB_0(r_1^2+r_2^2)}{4\tan (tB_0)}}\\
\times&\Big[e^{\frac{B_0r_1r_2}{2i\sin (tB_0)}\cos(\theta_1-\theta_2-tB_0)}
e^{-i\alpha(\theta_1-\theta_2-tB_0+2j_0\pi)}\chi(\theta, j_0)\\
&-\frac{\sin(\pi\alpha)}{\pi}\int_{\R} e^{-\frac{B_0r_1r_2}{2i\sin (tB_0)}\cosh s}\frac{e^{-\alpha s}}{1+e^{-s+i(\theta_1-\theta_2-tB_0)}}\,ds\Big],
\end{split}
\end{equation*}
which gives \eqref{S:express}.

\section{The Construction of Schr\"odinger propagator based on the Schulman-Sunada formula}\label{sec:const2}

In this section, we will use the Schulman-Sunada formula to reconstruct the Schr\"odinger propagator. We refer the reads to \cite{stov,stov1}.

Let $M=\R^2\setminus\{\vec{0}\}=(0,+\infty)\times \mathbb{S}^1$ where $\mathbb{S}^1$ is the unit circle. The universal covering space of $M$ is $\tilde{M}=(0,+\infty)\times \R$, then $M=\tilde{M}/\Gamma$ where the structure group $\Gamma=2\pi\Z$ acts in the second factor of the Cartesian product. Then Schulman's ansatz (see \cite{stov,stov1}) enables us to compute the Schr\"odinger propagator $e^{-it H_{\alpha, B_0}}$ on $M$ by using the Schr\"odinger propagator $e^{-it \tilde{H}_{\alpha, B_0}}$ (see the operator $\tilde{H}_{\alpha, B_0}$ in \eqref{t-H} below) on $\tilde{M}$. More precisely, see \cite[(1)]{stov1}, we have
\begin{equation}\label{SS}
e^{-it H_{\alpha, B_0}}(r_1,\theta_1; r_2, \theta_2)=\sum_{j\in\Z} e^{-it \tilde{H}_{\alpha, B_0}} (r_1,\theta_1+2j\pi; r_2, \theta_2).
\end{equation}
This is similar to the construction of wave propagator on $\mathbb{T}^n$, see \cite[(3.5.12)]{sogge}. In the following subsections, we will construct the Schr\"odinger propagator $e^{-it \tilde{H}_{\alpha, B_0}}$. \vspace{0.1cm}

\subsection{The eigenfunctions and eigenvalues}Before we construct the propagator $e^{-it \tilde{H}_{\alpha, B_0}} $, let us give some remarks about the difference and advantages between $\tilde{H}_{\alpha, B_0}$ and $H_{\alpha, B_0}$.
First, we recall the proof of Proposition \ref{prop:spect}. From \eqref{H-A},  in the polar coordinates $(r,\theta)\in M$,  then
\begin{equation*}
H_{\alpha, B_0}=-\partial_r^2-\frac{1}{r}\partial_r+\frac{1}{r^2}\Big(-i\partial_\theta+\alpha+\frac{B_0r^2}{2}\Big)^2,
\end{equation*}
which acts on $L^2(M, rdr\, d\theta)$. For $V_{k,m}(x)$ in \eqref{eigen-f} and $\lambda_{k,m}$ in \eqref{eigen-v},
we have show that
\begin{equation*}
H_{\alpha, B_0} V_{k,m}(x)=\lambda_{k,m} V_{k,m}(x).
\end{equation*}
We remark here that we choose $e^{ik\theta}$ as an eigenfunction of  the operator $\Big(-i\partial_\theta+\alpha+\frac{B_0r^2}{2}\Big)^2$ on $L^2_\theta([0,2\pi))$ which satisfies that
\begin{equation*}
\begin{cases}
\Big(-i\partial_\theta+\alpha+\frac{B_0r^2}{2}\Big)^2 \varphi(\theta)=\Big(k+\alpha+\frac{B_0r^2}{2}\Big)^2 \varphi(\theta)\\
\varphi(0)=\varphi(2\pi).
\end{cases}
\end{equation*}
Instead,  we consider the operator
\begin{equation}\label{t-H}
\tilde{H}_{\alpha, B_0}=-\partial_r^2-\frac{1}{r}\partial_r+\frac{1}{r^2}\Big(-i\partial_\theta+\alpha+\frac{B_0r^2}{2}\Big)^2,
\end{equation}
which acts on $L^2(\tilde{M}, rdr\, d\theta)$. We emphasize that the variable $\theta\in\R$ while not compact manifold $\mathbb{S}^1$. Then
we choose $e^{i(\tilde{k}-\alpha)\theta}$ as an eigenfunction of  the operator $\Big(-i\partial_\theta+\alpha+\frac{B_0r^2}{2}\Big)^2$ on $L^2_\theta(\R)$ which satisfies that
\begin{equation}\label{eq:ef}
\Big(-i\partial_\theta+\alpha+\frac{B_0r^2}{2}\Big)^2 \varphi(\theta)=\Big(\tilde{k}+\frac{B_0r^2}{2}\Big)^2 \varphi(\theta).
\end{equation}
It worths to point out that $\tilde{k}\in\R$ is a real number while $k\in\Z$. More important,  we inform move the $\alpha$ in right hand side of \eqref{eq:ef} to the $e^{i(\tilde{k}-\alpha)\theta}$ which will simplify the eigenfunctions. \vspace{0.2cm}

Now we modify the argument of  Proposition \ref{prop:spect}to solve
\begin{equation}\label{eigen-p'}
\tilde{H}_{\alpha, B_0} g(x)=\lambda g(x)
\end{equation}
to
 obtain the eigenfunctions of $\tilde{H}_{\alpha, B_0}$. Define the Fourier transform $F_{\theta\to\tilde{k}}$ with respect to the variable $\theta$
\begin{equation}\label{Fourier'}
F_{\theta\to\tilde{k}} f(r,\tilde{k})=\frac{1}{2\pi}\int_{\R}e^{i\tilde{k}\theta}f(r,\theta)\,d\theta:= \hat{f}(r,\tilde{k}),\quad \tilde{k}\in\R.
\end{equation}
By taking the Fourier transform of \eqref{eigen-p'}, in contrast to \eqref{eq:gk}
we obtain
\begin{equation}\label{eq:gk'}
\hat{g}''(r,\tilde{k})+\frac{1}{r}\hat{g}'(r,\tilde{k})-\frac{1}{r^2}\Big(\tilde{k}+\frac{B_0 r^2}{2}\Big)^2 \hat{g}(r,\tilde{k})=-\lambda \hat{g}(r,\tilde{k}).
\end{equation}
Let
\begin{equation*}
\psi_{\tilde{k}}(s)=\Big(\frac{2s}{B_0}\Big)^{-\frac{|\tilde{k}|}{2}}e^{\frac{s}{2}} \hat{g}(\Big(\sqrt{\frac{2s}{B_0}}, \tilde{k}\Big),
\end{equation*}
 then $\psi_{\tilde{k}}(s)$ satisfies
\begin{equation}\label{eq:psik}
s\psi''_{\tilde{k}}(s)+(1+|\tilde{k}|-s)\psi'_{\tilde{k}}(s)-\frac{1}{2}\Big(1+|\tilde{k}|+\tilde{k}-\frac{\lambda}{B_0}\Big)\psi_{\tilde{k}}(s)=0.
\end{equation}
which is same to \eqref{eq:phik} by replacing $|k+\alpha|$ by $\tilde{k}$.
By using Lemma \ref{lem:KCH} again, we have two linearly independent solutions of \eqref{eq:psik}, hence two linearly independent solutions of \eqref{eq:gk'} are given by
\begin{align*}
\hat{g}^1(\lambda;r,\tilde{k})&=r^{|\tilde{k}|}M\Big(\tilde{\beta}(\tilde{k},\lambda), \tilde{\gamma}(\tilde{k}),\frac{B_0r^2}{2}\Big)e^{-\frac{B_0r^2}{4}}\\
\hat{g}^2(\lambda;r,\tilde{k})&=r^{|\tilde{k}|}U\Big(\tilde{\beta}(\tilde{k},\lambda),\tilde{\gamma}(\tilde{k}),\frac{B_0r^2}{2}\Big)e^{-\frac{B_0r^2}{4}}
\end{align*}
with
\begin{align*}
\tilde{\beta}(\tilde{k},\lambda)&=\frac{1}{2}\Big(1+\tilde{k}+|\tilde{k}|-\frac{\lambda}{B_0}\Big),\\
\tilde{\gamma}(\tilde{k})&=1+|\tilde{k}|.
\end{align*}
Therefore, the general solution of \eqref{eq:psik} is given by
\begin{equation*}
\hat{g}(r,\tilde{k})=A_{\tilde{k}}\hat{g}^1(\lambda;r,\tilde{k})+B_{\tilde{k}}\hat{g}^2(\lambda;r,\tilde{k}).
\end{equation*}
where $A_{\tilde{k}}, B_{\tilde{k}}$ are two constants which depend on $\tilde{k}\in\R$.
Let
$$m:=-\tilde{\beta}(\tilde{k},\lambda)=-\frac{1}{2}\Big(1+\tilde{k}+|\tilde{k}|-\frac{\lambda}{B_0}\Big).$$
Similar as above,  we use this asymptotic Lemma \ref{lem:asy} to conclude that $m\in\mathbb{N}$ again, we omit the details.
\vspace{0.2cm}

Therefore, we must have
$$\mathbb{N} \ni m=-\frac{1}{2}\Big(1+\tilde{k}+|\tilde{k}|-\frac{\lambda}{B_0}\Big),$$
and
$$\hat{g}(r,\tilde{k})=r^{|\tilde{k}|}e^{-\frac{B_0r^2}{4}}\, P_{\tilde{k}-\alpha,m}\Big(\frac{B_0r^2}{2}\Big).$$
Therefore, we obtain a complete set of generalized eigenfunctions of $\tilde{H}_{\alpha, B_0}$
\begin{equation*}
\Big\{U_{m}(x, \tilde{k}): m\in \mathbb{N}, \tilde{k}\in\R\Big\}
\end{equation*}
where
\begin{equation}\label{eigen-f'}
U_{m}(x, \tilde{k})=|x|^{|\tilde{k}|}e^{-\frac{B_0 |x|^2}{4}}\, P_{\tilde{k}-\alpha,m}\Bigg(\frac{B_0|x|^2}{2}\Bigg)e^{i(\tilde{k}-\alpha)\theta}
\end{equation}
which belongs to $L^2(\tilde{M}, rdr\, d\theta)$.
Thus from $-m=\frac{1}{2}\Big(1+\tilde{k}+|\tilde{k}|-\frac{\lambda}{B_0}\Big)$, we solve \eqref{eigen-p'} to obtain the eigenvalues $\lambda$ of $\tilde{H}_{\alpha,B_0}$
\begin{equation*}
\lambda_{\tilde{k},m}=(2m+1+|\tilde{k}|+\tilde{k})B_0,\quad \tilde{k}\in\mathbb{R},\, m\in\mathbb{N}.
\end{equation*}
We obtain analogue of \eqref{Po-L} by using \eqref {La-po} and \eqref{P-L}
\begin{equation*}
\begin{split}
&\sum_{m=0}^\infty e^{-cm}\frac{m !}{\Gamma(m+|\tilde{k}|+1)}
\Bigg(\begin{array}{c}
    m+|\tilde{k}| \\
    m \\
  \end{array}
\Bigg)^2 P_{\tilde{k}-\alpha,m} (a) P_{\tilde{k}-\alpha,m} (b)\\
&=\frac{e^{\frac{|\tilde{k}| c}2}}{(ab)^{\frac{|\tilde{k}|}2}(1-e^{-c})} \exp\left(-\frac{(a+b)e^{-c}}{1-e^{-c}}\right) I_{|\tilde{k}|}\left(\frac{2\sqrt{ab}e^{-\frac{c}2}}{1-e^{-c}}\right).
\end{split}
\end{equation*}

\subsection{Construction of $e^{-it \tilde{H}_{\alpha, B_0}} $} Now we construct the propagator $e^{-it \tilde{H}_{\alpha, B_0}} $ by using the above eigenfunctions.
Let $\tilde{U}_{m}(x, \tilde{k})$ be the $L^2$-normalization of $U_{m}(x, \tilde{k})$ in \eqref{eigen-f'}.
We wirte initial data $f(x)\in L^2$ as
\begin{equation*}
f(x)=\sum_{m\in\mathbb{N}}\int_{\R} c_{m}(\tilde{k}) \tilde{U}_{m}(x, \tilde{k}) \, d\tilde{k}
\end{equation*}
where
\begin{equation*}
c_{m}(\tilde{k})=\int_{\mathbb{R}^2}f(x)\overline{\tilde{U}_{m}(x, \tilde{k})}\, dx.
\end{equation*}
Using the Fourier transform \eqref{Fourier'}, we solve
\begin{equation*}
\begin{cases}
\big(i\partial_t-\tilde{H}_{\alpha,B_0}\big)u(t,x)=0,\quad (t, x)\in \R\times \tilde{M}\\
u(0,x)=f(x).
\end{cases}
\end{equation*}
to obtain
\begin{equation*}
u(t,x)=\sum_{ m\in\mathbb{N}} \int_{\R} e^{-it\lambda_{\tilde{k},m}}\left(\int_{\mathbb{R}^2}f(y)\overline{\tilde{U}_{m}(y, \tilde{k})}\, dy\right)\tilde{U}_{m}(x, \tilde{k}) \, d\tilde{k}.
\end{equation*}
By repeating
the the proof of Proposition \ref{prop:S}, we similarly show that
\begin{equation*}
\begin{split}
u(t,x)
=&\frac{B_0}{8\pi^2 i\sin (tB_0)}\int_0^\infty \int_{\R}e^{-\frac{B_0(r_1^2+r_2^2)}{4i\tan (tB_0)}} \\&\times \int_{\R}
\Bigg( e^{i(\tilde{k}-\alpha)(\theta_1-\theta_2-tB_0)}
I_{|\tilde{k}|}\bigg(\frac{B_0r_1r_2}{2i\sin (tB_0)}\bigg)\Bigg)\, d\tilde{k} f(r_2,\theta_2)r_2dr_2 d\theta_2.
\end{split}
\end{equation*}
In contrast to \eqref{S:sum}, the difference here is that we  replace the summation in $k\in\Z$ by integration on $\tilde{k}\in\R$.
Hence the kernel of $e^{-it \tilde{H}_{\alpha, B_0}} $ is
\begin{equation*}
\begin{split}
\tilde{K}_S(x,y)
=&\frac{B_0}{4\pi i\sin (tB_0)}e^{-\frac{B_0(r_1^2+r_2^2)}{4i\tan (tB_0)}} \\&\times \int_{\R}
\Bigg( e^{i(\tilde{k}-\alpha)(\theta_1-\theta_2-tB_0)}
I_{|\tilde{k}|}\bigg(\frac{B_0r_1r_2}{2i\sin (tB_0)}\bigg)\Bigg)\, d\tilde{k},
\end{split}
\end{equation*}
where $x=(r_1,\theta_1)\in \tilde{M}$ and $y=(r_2,\theta_2)\in \tilde{M}$.

Now, instead of summing in $k$ as before,
 we consider the integration in $\tilde{k}$. Again by letting $z=\frac{B_0r_1r_2}{2i\sin (tB_0)}$ and $\theta=\theta_1-\theta_2-tB_0$ and
 using \eqref{mBessel:integral}, we compute that
\begin{equation*}
\begin{split}
& \frac1\pi \int_{\R} e^{i\tilde{k}\theta} \int_0^\pi e^{z\cos s} \cos(|\tilde{k}| s) dsd\tilde{k}\\
 &=\frac1{2\pi} e^{z\cos \theta} \Big(\chi_{[0,\pi]}(\theta)+\chi_{[0,\pi]}(-\theta)\Big)= e^{z\cos \theta} \chi_{[-\pi,\pi]}(\theta),
 \end{split}
\end{equation*}
  and
  \begin{equation*}
    \begin{split}
& \frac1\pi \int_{\R} e^{i\tilde{k}\theta} \sin(\pi |\tilde{k}|)\int_0^\infty e^{-z\cosh s} e^{-s|\tilde{k}|} ds \,d\tilde{k}\\
 &=  \frac1\pi\int_0^\infty e^{-z\cosh s} \Big( \int_{0}^\infty e^{i\tilde{k}\theta}\frac{e^{i\pi \tilde{k}}-e^{-i\pi \tilde{k}}}{2i}  e^{-s\tilde{k}} d\tilde{k} \\
& \qquad+
   \int_{-\infty}^0 e^{i\tilde{k}\theta} \frac{e^{-i\pi \tilde{k}}-e^{i\pi \tilde{k}}}{2i}  e^{s\tilde{k}} d\tilde{k}\Big) \,ds\\
  &=  \frac1{2\pi}\int_{-\infty}^\infty e^{-z\cosh s} \big(\frac{1}{\theta+\pi+is}-\frac{1}{\theta-\pi+is}\big) ds
  \end{split}
  \end{equation*}
  Therefore, we obtain
  \begin{equation*}
\begin{split}
&\tilde{K}_S(x,y)
=\frac{B_0}{8\pi^2 i\sin (tB_0)}e^{-\frac{B_0(r_1^2+r_2^2)}{4i\tan (tB_0)}} e^{-i\alpha\theta}\\&\times
\Big( e^{z\cos \theta} \chi_{[-\pi,\pi]}(\theta)-\frac1{2\pi}\int_{-\infty}^\infty e^{-z\cosh s} \big(\frac{1}{\theta+\pi+is}-\frac{1}{\theta-\pi+is}\big) ds
\Big).
\end{split}
\end{equation*}
Finally, by using \eqref{SS}, we have that
\begin{equation*}
\begin{split}
&e^{-it H_{\alpha, B_0}}(r_1,\theta_1; r_2, \theta_2)\\
&=\frac{B_0}{8\pi^2 i\sin (tB_0)}e^{-\frac{B_0(r_1^2+r_2^2)}{4i\tan (tB_0)}}
\sum_{j\in\Z} e^{-i\alpha(\theta+2j\pi)}
\Big(  e^{z\cos (\theta+2j\pi)} \chi_{[-\pi,\pi]}(\theta+2j\pi)\\&\qquad-\frac1{2\pi}\int_{-\infty}^\infty e^{-z\cosh s} \big(\frac{1}{(\theta+2j\pi)+\pi+is}-\frac{1}{(\theta+2j\pi)-\pi+is}\big) ds
\Big).
\end{split}
\end{equation*}
Due to the period property of $\cos$ function and for all $\theta\in\mathbb{R},$ there exists $j_0\in\mathbb{Z}$ such that $\theta+2j_0\pi\in [-\pi,\pi),$ the first term in the big bracket becomes
\begin{align*}
&e^{z\cos \theta} e^{-i\alpha\theta}  \sum_{j\in\Z} e^{-i\alpha 2j\pi}
\chi_{[-\pi,\pi]}(\theta+2j\pi)\\
&=e^{z\cos \theta} e^{-i\alpha\theta}e^{-i2\pi j_0\alpha}\times\left\{
    \begin{array}{ll}
      1, & \hbox{$|\theta+2j_0\pi|<\pi$;} \\
      e^{-i2\pi\alpha}+1, & \hbox{$\theta+2j_0\pi=-\pi$;}\\
      e^{i2\pi\alpha}+1, & \hbox{$\theta+2j_0\pi=\pi$.}
    \end{array}
  \right.
\end{align*}
which is the same to \eqref{term1+2}. Hence it is the same to the first term of \eqref{S:express}.

For the second term in the big bracket, we use the formula
\begin{equation*}
\sum_{j\in\Z} \frac{e^{-2\pi i\alpha j}}{\sigma+2\pi j}=\frac{i e^{i\alpha\sigma}}{e^{i\sigma}-1},\quad \alpha\in(0,1),\quad \sigma\in\C\setminus 2\pi\Z,
\end{equation*}
to obtain
\begin{equation*}
\begin{split}
\sum_{j\in\Z} e^{-2\pi i\alpha j}&\big(\frac{1}{(\theta+2j\pi)+\pi-is}-\frac{1}{(\theta+2j\pi)-\pi-is}\big)\\&
=2\sin(\pi \alpha)\frac{e^{\alpha(s+i\theta)}}{1+e^{s+i\theta}}.
\end{split}
\end{equation*}
Now we consider the second term
\begin{equation*}
\begin{split}
&-e^{-i\theta\alpha}\frac{2\sin(\pi\alpha)}{2\pi}e^{i\theta\alpha} \int_{-\infty}^\infty e^{-z\cosh s} \frac{e^{\alpha s}}{1+e^{s+i\theta}}\,ds\\
=&-\frac{\sin(\pi\alpha)}{\pi}\int_{-\infty}^\infty e^{-z\cosh s} \frac{e^{-\alpha s}}{1+e^{-s+i\theta}}\,ds
\end{split}
\end{equation*}
Recall $\theta=\theta_1-\theta_2-tB_0$ and $z=\frac{B_0r_1r_2}{2i\sin (tB_0)}$, we
obtain
\begin{equation*}
\begin{split}
K_S(x,y)&=\frac{B_0e^{-itB\alpha}}{8\pi^2 i\sin (tB_0)}e^{\frac{iB_0(r_1^2+r_2^2)}{4\tan (tB_0)}}\\
\times&\Big[e^{\frac{B_0r_1r_2}{2i\sin (tB_0)}\cos(\theta_1-\theta_2-tB_0)}
e^{-i\alpha(\theta_1-\theta_2-tB_0+2j_0\pi)}\chi(\theta, j_0)\\
&-\frac{\sin(\pi\alpha)}{\pi}\int_{\R} e^{-\frac{B_0r_1r_2}{2i\sin (tB_0)}\cosh s}\frac{e^{-\alpha s}}{1+e^{-s+i(\theta_1-\theta_2-tB_0)}}\,ds\Big],
\end{split}
\end{equation*}
which is exact same to \eqref{S:express}.
\end{proof}

\section{proof of Theorem \ref{thm:S}}\label{sec:proof}
In this section, we prove the main Theorem \ref{thm:S} by using \eqref{S:express}. We first prove the dispersive estimate \eqref{dis-S}. To this end, it is enough to prove
\begin{align*}
\Big|\int_{\R} e^{-\frac{B_0r_1r_2}{2i\sin (tB_0)}\cosh s}\frac{e^{-\alpha s}}{1+e^{-s+i(\theta_1-\theta_2-tB_0)}}\,ds\Big|\leq C,
\end{align*}
where $C$ is a constant independent of $t$, $r_1, r_2$ and $\theta_1, \theta_2$.

We notice that
\begin{align*}
&\int_{\R} e^{-\frac{B_0r_1r_2}{2i\sin (tB_0)}\cosh s}\frac{e^{-\alpha s}}{1+e^{-s+i(\theta_1-\theta_2-tB_0)}}\,ds\\
&=\int_{0}^\infty e^{-\frac{B_0r_1r_2}{2i\sin (tB_0)}\cosh s}\Big(\frac{e^{-\alpha s}}{1+e^{-s+i(\theta_1-\theta_2-tB_0)}}+\frac{e^{\alpha s}}{1+e^{s+i(\theta_1-\theta_2-tB_0)}}\Big)\,ds
\end{align*}
then we just need to verify that, for $\theta=\theta_1-\theta_2-tB_0,$
\begin{align*}
\int_{0}^\infty\Big|\frac{e^{-\alpha s}}{1+e^{-s+i\theta}}+\frac{e^{\alpha s}}{1+e^{s+i\theta}}\Big| \,ds\lesssim 1,
\end{align*}
where the implicit constant is independent of $\theta$.
In fact,
\begin{align*}
&\frac{e^{-\alpha s}}{1+e^{-s+i\theta}}+\frac{e^{\alpha s}}{1+e^{s+i\theta}}\\
&=\frac{\cosh(\alpha s)e^{-i\theta}+\cosh((1-\alpha)s)}{\cos\theta+\cosh s}\\
&=\frac{\cosh(\alpha s)\cos\theta+\cosh((1-\alpha)s)-i\sin\theta\cosh(\alpha s)}{2(\cos^2(\frac{\theta}{2})+\sinh^2(\frac s 2))}\\
&=\frac{2\cos^2(\frac{\theta}2)\cosh(\alpha s)+(\cosh((1-\alpha)s)-\cosh(\alpha s))-2i\sin(\frac\theta 2)\cos(\frac\theta 2)\cosh(\alpha s)}{2(\cos^2(\frac{\theta}{2})+\sinh^2(\frac s 2))}.
\end{align*}
Since that $\cosh x-1\sim\frac{x^2}2,\sinh x\sim x$,  as $x\to 0;$ $\cosh x\sim e^x,\sinh x\sim e^{x}$, as $x\to \infty,$
 we have
\begin{align*}
\int_{0}^\infty\Big|\frac{\cos^2(\frac{\theta}2)\cosh(\alpha s)}{\cos^2(\frac{\theta}{2})+\sinh^2(\frac s 2)}\Big| \,ds\lesssim\int_0^1\frac{2|\cos(\frac\theta 2)|}{s^2+(2|\cos(\frac\theta 2)|)^2}ds+\int_1^\infty\ e^{(\alpha-1)s}ds\lesssim 1.
\end{align*}
Similarly,  we obtain
\begin{align*}
\int_{0}^\infty\Big|\frac{\sin(\frac\theta 2)\cos(\frac\theta 2)\cosh(\alpha s)}{\cos^2(\frac{\theta}{2})+\sinh^2(\frac s 2)}\Big| \,ds\lesssim 1.
\end{align*}
Finally, we verify that
\begin{align*}
&\int_{0}^\infty\Big|\frac{\cosh((1-\alpha)s)-\cosh(\alpha s)}{\cos^2(\frac{\theta}{2})+\sinh^2(\frac s 2)}\Big| \,ds\\
&\lesssim\int_0^1\frac{|\frac{(1-\alpha)^2}2-\frac{\alpha^2}2|s^2}{s^2}ds+\int_1^\infty \big( e^{-\alpha s}+e^{(\alpha-1)s} \big)ds\lesssim 1.
\end{align*}
Therefore, we obtain the bound \eqref{dis-S} as desired.

Next we prove \eqref{str-S}. For $T\in (0,\frac{\pi}{2B_0})$ and $t\in (0,T)$, we have
\begin{equation*}
\frac{2}{\pi}\leq\frac{\sin (tB_0)}{tB_0}\leq 1,
\end{equation*}
then the dispersive estimate \eqref{dis-S} gives
\begin{equation}\label{dis-S-0}
\|e^{-itH_{\alpha,B_0}}\|_{L^1(\mathbb{R}^2)\rightarrow L^\infty(\mathbb{R}^2)}\lesssim\frac{1}{t},\quad \forall t\in(0,T],
\end{equation}
when $T\in (0,\frac{\pi}{2 B_0}]$. On the other hand, by the spectral theory, we have the energy estimate
\begin{equation}\label{energy}
\|e^{-itH_{\alpha,B_0}}\|_{L^2(\mathbb{R}^2)\rightarrow L^2(\mathbb{R}^2)}\lesssim 1.
\end{equation}

To prove the Strichartz estimates \eqref{str-S}, we recall
the abstract mechanism of Keel-Tao \cite{KT}

\begin{theorem}(Keel-Tao \cite{KT})
Let $(X,d\mu)$ be a measure space and $H$ a Hilbert space. Suppose
that for each time $t\in\mathbb{R}$, $U(t): H \rightarrow L^2(X)$
which satisfies the energy estimate
\begin{equation*}
\|U(t)\|_{H\rightarrow L^2}\leq C,\quad t\in \mathbb{R}
\end{equation*}
and that for some $\sigma>0$ either
\begin{equation*}
\|U(t)U(s)^*f\|_{L^\infty(X)}\leq C|t-s|^{-\sigma}\|f\|_{L^1(X)},\quad t\neq s.
\end{equation*}
Then the
estimates  hold
\begin{equation*}
\|U(t)f\|_{L^q_tL^r(X)}\lesssim \|f\|_{L^2(X)}.
\end{equation*}
where $$(q,r)\in\Lambda:=\Big\{(q,r)\in[2,+\infty]\times[2,+\infty): \frac2q=n(\frac12-\frac1r)\Big\}.$$
\end{theorem}
By making use of this theorem to $U(t)=e^{-itH_{\alpha,B_0}}\chi_{[0,T]}(t)$, since \eqref{dis-S-0} and \eqref{energy}, we obtain \eqref{str-S}.
Therefore we finish the proof of Theorem \ref{thm:S}.

\begin{center}

\end{center}

\end{document}